\documentclass[11pt]{amsart}
\usepackage{mathtools}
\usepackage{amsmath}
\usepackage{amsthm}
\usepackage[inline]{enumitem}
\usepackage{amssymb}
\usepackage{mathrsfs}
\usepackage[margin=1.25in]{geometry}
\usepackage{amscd}
\usepackage{import}
\usepackage{pdfpages}
\usepackage{transparent}
\usepackage{xcolor}
\usepackage[colorlinks=true]{hyperref}
\usepackage{cleveref}

\newtheorem{thm}{Theorem}
\newtheorem{lem}[thm]{Lemma}
\newtheorem{prop}[thm]{Proposition}
\newtheorem{cor}[thm]{Corollary}
\theoremstyle{definition}

\crefname{ex}{example}{examples}

\DeclareMathOperator{\map}{Map}
\DeclareMathOperator{\stab}{Stab}
\DeclareMathOperator{\out}{Out}
\DeclareMathOperator{\diff}{Diff}
\DeclareMathOperator{\homeo}{Homeo}
\newcommand{\R}{\mathbb{R}}
\newcommand{\Z}{\mathbb{Z}}
\newcommand{\sym}{\operatorname{Sym}(\mathbb{N})}

\begin{document}
\title{Nielsen Realization for Infinite-Type Surfaces}

\author{Santana Afton}
\address{School of Mathematics\\
	Georgia Institute of Technology\\
	Atlanta, Georgia, USA}
\email[S.~Afton]{santana.afton@gatech.edu}

\author{Danny Calegari}
\address{Department of Mathematics\\
	University of Chicago\\ 
	Chicago, Illinois, USA}
\email[D.~Calegari]{dannyc@math.uchicago.edu}

\author{Lvzhou Chen}
\address{Department of Mathematics\\ 
	University of Chicago\\ 
	Chicago, Illinois, USA}
\email[L.~Chen]{lzchen@math.uchicago.edu}

\author{Rylee Alanza Lyman}
\address{Department of Mathematics\\ 
	Rutgers University at Newark\\ 
	Newark, New Jersey, USA}
\email[R.~Lyman]{rylee.lyman@rutgers.edu}

\begin{abstract}
    Given a finite subgroup $G$ of the mapping class group
    of a surface $S$, 
    the Nielsen realization problem
    asks whether $G$
    can be realized as a finite group of homeomorphisms of $S$.
	In 1983, Kerckhoff showed that for $S$ a finite-type surface,
	any finite subgroup $G$ 
	may be realized as a group of isometries of some hyperbolic metric on $S$.
    We extend Kerckhoff's result to
    orientable, infinite-type surfaces. 
    As applications, 
	we classify torsion elements in the mapping class group of the plane minus a Cantor set, 
	and also show that topological groups
    containing sequences of torsion elements limiting to the identity 
	do not embed continuously into the mapping class group of $S$.
	Finally, we show that compact subgroups of the mapping class group of $S$ are finite,
	and locally compact subgroups are discrete.
\end{abstract}

\maketitle

In 1932, Nielsen asked whether finite subgroups of
mapping class groups act on surfaces.
In 1983, Kerckhoff~\cite{Kerckhoff} gave the following
strong affirmative answer.

\begin{thm}[Kerckhoff, Theorem 5~\cite{Kerckhoff}]\label{Kerckhoffthm}
	Let $S$ be a finite-type surface
	with negative Euler characteristic. 
	Any finite subgroup of the mapping class group of $S$
	may be realized
	as a group of isometries of some hyperbolic metric on $S$.
\end{thm}

Let $S$ be a surface. We distinguish two kinds of surfaces,
saying $S$ is a \emph{finite-type} surface
if its fundamental group is finitely generated,
and is an \emph{infinite-type} surface otherwise.
Recently there has been a surge of interest in infinite-type surfaces 
and their mapping class groups.
We refer the interested reader 
to a recent survey by Aramayona and Vlamis \cite{AramayonaVlamis}.

The main purpose of this paper is to extend Kerckhoff's result
to the infinite-type case.

\begin{thm}\label{maintheorem}
	Let $S$ be an orientable, infinite-type surface
	whose boundary is either empty or a union of circles.
	Any finite subgroup of the mapping class group of $S$
	may be realized 
	as a group of isometries of some hyperbolic metric on $S$.
\end{thm}

Let us briefly sketch the idea of the proof.
Let $G$ be a finite subgroup of the mapping class group of $S$.
The mapping class group acts on the \emph{Teichm\"uller space} of $S$,
denoted $\mathscr{T}(S)$,
which parameterizes hyperbolic structures on $S$ up to isotopy.
Kerckhoff defines a $G$-invariant map $\ell_G\colon \mathscr{T}(S) \to \mathbb{R}_+$.
The map sends a hyperbolic structure to the sum
of the geodesic lengths of a certain finite $G$-invariant collection of 
simple closed curves on $S$.
When $S$ is of finite type,
Kerckhoff proves that the map $\ell_G$ attains a unique minimum.
This minimum is fixed by $G$, yielding an action of $G$ by isometries
of the corresponding  hyperbolic structure on $S$.

In the infinite-type setting,
all of the tools in Kerckhoff's proof are available to us,
but the sum defining $\ell_G$ diverges. 
Instead, we find an exhaustion of $S$ by connected,
$G$-invariant (homotopy classes of)
finite-type subsurfaces $S_0 \subset S_1 \subset \dotsb$.
Kerckhoff's theorem applies to each piece $\overline{S_k\setminus S_{k-1}}$,
and we show how to assemble the pieces to give a hyperbolic structure on $S$
and an action of $G$ by isometries.
It would be interesting to know if Kerckhoff's method of proof
could be applied more directly.

We remark that if $\Sigma$ is a finite-type subsurface
of $S$, the hyperbolic metric on $S$ in the theorem is
chosen so that $\Sigma$ inherits a hyperbolic metric of
finite volume. In other words, each isolated end of $S$
is given a metric modeled on the pseudosphere rather
than a flared annulus.

Let us mention recent work of Aougab, Patel and Vlamis.
In \cite{AougabPatelVlamis}, they study a similar realization problem:
given a group $G$ and a surface $S$,
is there a hyperbolic metric on $S$ whose isometry group is isomorphic to $G$?
They show that for many infinite-type surfaces, every countable group $G$ 
\emph{can} in fact be realized in this way.
Of course, note that their result does not preclude the existence of other embeddings of $G$ 
into $\map(S)$ which cannot be realized as groups of isometries.

\begin{cor}
    If $S$ is an orientable, infinite-type surface
    with nonempty compact boundary,
	the relative mapping class group fixing the boundary pointwise
    is torsion-free.
\end{cor}

As an application, the Nielsen realization theorem allows us to classify torsion elements 
in the mapping class group of the plane minus a Cantor set; see Theorem \ref{thm: classify torsion elements}.

Equip the full homeomorphism group of $S$ with the compact-open topology,
and the mapping class group of $S$ with the quotient topology.
As another application, we have the following two theorems.

\begin{thm}
	If $G$ is a topological group containing a sequence of nontrivial
	finite order elements limiting to the identity, then
	$G$ does not embed (as a topological group) in the mapping class group of $S$.
\end{thm}

\begin{thm}
	Compact subgroups of $\map(S)$ are finite,
	and locally compact subgroups are discrete.
\end{thm}

The authors thank Kathryn Mann
for suggesting that it would be interesting 
to find a topological proof of the former theorem
and thank Mladen Bestvina for suggesting the latter theorem.
We also thank Carolyn Abbott for helpful conversations.

In the remainder of this introduction,
let us make the statement of \Cref{maintheorem} more precise.
The \emph{mapping class group} of $S$, denoted $\map(S)$, 
is the group $\pi_0(\homeo(S))$ of homotopy classes of homeomorphisms
$f\colon S \to S$.
If $S$ has nonempty boundary $\partial S$,
let $\map(S,\partial S)$ denote the group of homotopy classes 
of homeomorphisms $f\colon S \to S$
where we require all homeomorphisms and homotopies 
to fix the boundary $\partial S$ pointwise.

A classification of infinite-type surfaces was given by
Ker\'ekj\'art\'o \cite{Kerekjarto} and Richards \cite{Richards}.
We recall that each surface $S$ has a \emph{space of ends,}
defined as usual as an inverse limit 
$\varprojlim\pi_0(S\setminus K)$ as $K$ ranges over the compact subsets of $S$.
If we fix a compact exhaustion $K_0 \subset K_1 \subset \dotsb$ of $S$,
an \emph{end} is represented as a sequence
\[U_0 \supset U_1 \supset \dotsb,\]
where each $U_i$ is a connected component of $S \setminus K_i$.
The end is \emph{isolated} if all but finitely many of the $U_i$ have one end,
and \emph{planar} if all but finitely many of the $U_i$ are planar.
These properties do not depend on the choice of compact exhaustion.
Finite-type surfaces have finitely many isolated planar ends.

By a \emph{hyperbolic structure} on a surface $S$,
we mean $S$ equipped with a complete Riemannian metric of constant curvature $-1$
satisfying the following two conditions.
\begin{enumerate*}[label=(\roman*)]
	\item We require that each isolated planar end of $S$
		is modeled on the pseudosphere,
		and following Kerckhoff \cite[Section 4]{Kerckhoff},
	\item we require that each boundary curve is a geodesic with length $1$.
\end{enumerate*}
In the case where $S$ is of finite type,
the condition on ends is satisfied 
by insisting that the metric has finite volume.
We remark that the condition that boundary curves have length $1$
is slightly nonstandard but useful. 
In addition, forcing this condition on our finite-type exhaustion of $S$ places the
hyperbolic structure on $S$ in the component containing the ``thick part'' of Teichm\"uller space.

\section{Finding an Invariant Exhaustion}
Fix $S$ an orientable, connected, infinite-type surface
whose boundary is empty or a union of circles
and fix $G$ a nontrivial finite subgroup of $\map(S)$.
The goal of this section is to prove the following proposition.

\begin{prop}\label{exhaustion} 
	There exists an exhaustion of $S$ by connected, 
	finite-type subsurfaces $\varnothing=S_0 \subset S_1 \subset \dotsb$
	such that for each $k\ge 1$, $S_k$ is $G$-invariant up to homotopy,
	%with injective induced homomorphism $\rho_k: G\to\map(S_k)$,
	and each component of $S_k \setminus S_{k-1}$ 
	has negative Euler characteristic.
\end{prop}

This proposition is crucial to our proof of \Cref{maintheorem}. 

\begin{proof} First, fix an arbitrary exhaustion $\{K_i\}$ of $S$ 
	by connected, finite-type subsurfaces, e.g.
	coming from a pants decomposition of $S$.
	Let $X$ be a subset of $\diff(S)$ 
	containing a single representative from each mapping class in $G$. 
	We will proceed by induction 
	on the length of a chain of subsurfaces $S_0\subset S_1\subset \dotsb \subset S_n$. 
	For the base case, set $S_0 = \varnothing$.

	Since $S$ is of infinite type,
	there exists some essential, simple closed curve $\alpha$ on $S\setminus S_n$.
	Recall that a simple closed curve is \emph{essential}
	if it does not bound a disk 
	nor is homotopic to a boundary curve or a puncture. 
	Moreover, there exists some term $K_{i_n}$ of the exhaustion containing 
	both the $X$-orbit of $\alpha$ and $S_n$, since $X$ is finite. 
	To the end of creating a $G$-invariant exhaustion, consider
	\[
		K = \bigcap_{f\in X} f(K_{i_n}).
	\]
	Note that $K$ contains the $X$-orbit of $\alpha$ as well as $S_n$. 
	By performing a preliminary isotopy,
	we may assume the $f(K_{i_n})$ pairwise intersect transversely,
	and thus $K$ is itself a subsurface. 
	Since $f(K) = K$ for each $f\in X$ by construction,
	we may take $S_{n+1}$ to be $K$, 
	together with the finitely many components of $S\setminus K$ 
	that contain no essential, simple closed curve.
	If any component of $S_{n+1} \setminus S_n$ is an annulus,
	by induction we may assume that the component of $S\setminus S_n$
	it determines contains an essential simple closed curve.
	Thus there is no loss in replacing $S_{n+1}$ by an isotopic subsurface
	so that no component of $S_{n+1}\setminus S_n$ is an annulus.
	
	Choosing the essential simple loop $\alpha$ appropriately at each step, 
	the inductive construction above gives an exhaustion of $S$.

%	It remains to make the induced maps $\rho_k$ injective. Since $S_k$ is $G$-invariant up to homotopy, 
%	each $g\in G$ gives a well-defined mapping class $\rho_k(g)$ in $\map(S_k)$ by choosing an arbitrary homotopy between $S_k$ and $x_g(S_k)$,
%	where $x_g$ is a diffeomorphism on $S$ representing $g$.
%	We may fix finitely many loops $\{\beta_j\}$ on $S$ to witness the nontriviality of each $g\neq id \in G$. 
%	We can choose each subsurface $K_{i_n}$ in the construction of $S_1$ to further contain $\{\beta_j\}$. 
%	This ensures $\rho_1$ to be injective, and thus each $\rho_k$ is injective since its restriction $\rho_1$ is.	
\end{proof}

\section{Nielsen Realization}
We continue with the notation from the previous section.
Choose a hyperbolic structure on $S$ so that each
$S_i$ has finite volume and each component of $\partial S_i$
is a geodesic with length $1$.
Write $P_i = \overline{S_i \setminus S_{i-1}}$.
Since each $S_i$ is $G$-invariant up to isotopy, so is each $P_i$. 
Each $g\in G$ gives a well-defined mapping class in $\map(P_i)$ 
by choosing an arbitrary homotopy between $P_i$ and $x_g(P_i)$,
where $x_g$ is a diffeomorphism on $S$ representing $g$.
Thus each $g\in G$ defines a mapping class $\rho(g)$ in $\map(P_i)$.
The inclusion $G\to \map(P_i)$ is actually injective and follows from
the proof below (see e.g. Proposition \ref{prop: torsion}). 
Our proof does not depend on this fact.

%In this way $G$ descends to a \emph{subgroup} of each $\map(P_i)$. 
%Indeed, if $\rho(g)$ is the trivial mapping class on $P_i$ for some $g\in G$, 
%then $g$ is a torsion mapping class in $\map(S_i)$ that is trivial on a subsurface $P_i$ with $\chi(P_i)<0$.
%By the classical Nielsen realization theorem on $S_i$,
%we note that $g$ must be trivial in $\map(S_i)$ in this case, 
%which only occurs when $g=1$ since $\rho_i:G\to\map(S_i)$ is injective by construction. 
%Thus $\rho: G\to \map(P_i)$ is injective.

\begin{proof}[Proof of \Cref{maintheorem}]
According to Kerckhoff, 
there is another hyperbolic structure on each $P_i$
with respect to which the image of $G$ in $\map(P_i)$ may be realized as a group of isometries
\cite[Theorem 5, discussion following Theorem 4]{Kerckhoff}.
Note that each boundary curve of each $P_i$ remains geodesic with length $1$.

Let us say a few words about the case where $P_i$ is disconnected.
Let $C$ be a component of $P_i$, 
and write $H$ for the stabilizer of $C$ in $G$.
By the above, there is a hyperbolic structure on $C$
with an action of $H$ by isometries.
Choose a set of coset representatives for $G/H$.
If $C'$ is a component of the orbit of $C$ distinct from $C$ choose a representative in $X$ taking $C'$ to $C$,
and give $C'$ a new hyperbolic structure by pulling back
the metric from $C$.
Proceeding orbit by orbit, this  yields a realization of $G$ on $P_i$.

We now have a hyperbolic structure for each $P_i$ 
with an action of $G$ by isometries.
Gluing up the $P_i$ via the identifications coming from $S$
yields a hyperbolic structure on $S$,
but we need to do so respecting the action of $G$.
As above, we glue the boundary curves shared by $P_i$ and $S_{i-1}$ orbit by orbit.
It suffices to show that for each boundary curve $c$ in $P_i \cap S_{i-1}$
there is a gluing that identifies 
the two circle actions on $c$ by its stabilizer $\stab(c)$.
Note that each $g\in\stab(c)$ acts on $c$ by rigid rotations, 
and with the (opposite) orientations induced from $P_i$ and $S_{i-1}$,
it suffices to show that the angles of rotations for the two actions differ by a negative sign. 
Indeed, on the $P_i$ (resp. $S_{i-1}$) side, the angle is the rotation number of the $g$ action
on the circle of geodesic rays on $P_i$ (resp. $S_i$) starting from and perpendicular to $c$.
This action is similar to the one on the conical circle (see e.g. \cite{Bavard_Walker_2}),
and only depends on the mapping class $g$. Thus the two angles differ by a negative sign
due to the opposite orientations.

As a result, we have an isometric action of $G$ on each $P_i$
that respects the gluing,
yielding a hyperbolic structure and an isometric action of $G$ on $S$.
\end{proof}

\section{Classification of torsion elements}\label{sec: classify}
If $S$ is a surface of finite genus and empty boundary, 
then by the classification theorem \cite{Richards} it is realized as $\Sigma- E$, 
where $\Sigma$ is the closed surface with the same genus as $S$ and $E$ is a totally disconnected closed subset of $\Sigma$
homeomorphic to the space of ends. In this case, Theorem~\ref{maintheorem} 
implies that any finite subgroup $G$ of 
$\map(S)$ is realized by some $G$-action on $\Sigma$ by homeomorphisms preserving $E$. This is because 
$\homeo^+(\Sigma-E)\cong \homeo^+(\Sigma,E)$, where the latter denotes orientation-preserving
homeomorphisms of $\Sigma$ preserving $E$.

In particular, one can use Theorem \ref{maintheorem} to classify torsion elements. Here we focus on an
example, the case of $S=\R^2-K$, where $K$ is a Cantor set. In this situation, the mapping class group
acts faithfully on the \emph{conical circle} $S^1_C$ consisting of geodesics (for a fixed complete
hyperbolic metric on $S$) emanating from $\infty$; see e.g. \cite{Bavard_Walker_1, Calegari_Chen}. Thus each mapping
class $g$ has a rotation number, which can be read off from its action on a special subset of $S^1_C$,
namely the \emph{short rays}, which are proper simple geodesics connecting $\infty$ to some point in the Cantor set.

\begin{thm}\label{thm: classify torsion elements}
	Let $S=\R^2-K$, where $K$ is a Cantor set. 
	Finite order elements of $\map(S)$ fix at most one point in $K$. 
	For each $n\ge 2$, elements in $\map(S)$ of order $n$ fall 
	into $2\varphi(n)$ conjugacy classes, which are distinguished by the rotation number and whether the 
	element fixes exactly one point in $K$ or none. Here $\varphi(n)$ is the number of positive integers up 
	to $n$ that are coprime to $n$.
\end{thm}
\begin{proof}
	Let $g\in\map(S)$ be an element of order $n$. Then the action of $g$ on the conical circle $S^1_C$ has rotation
	number $m/n\mod \Z$ for some $m$ coprime to $n$. %rot is hom on finite cyclic subgroups
	By Theorem \ref{maintheorem}, we can realize $g$
	as some $\tilde{g}\in \homeo^+(S^2,K\cup\{\infty\})$ of order $n$. It is known that any finite order 
	homeomorphism on $S^2$ is conjugate to a rigid rotation \cite{Zimmermann} and the quotient $S^2/\tilde{g}$ 
	is still homeomorphic to $S^2$. Considering the rotation number, we conclude that $\tilde{g}$ is conjugate 
	to a rigid rotation by $2m\pi/n$, and there is exactly one fixed point $p\in S^2$ other than $\infty$.
	
	We can put a Cantor set on $S^2$ invariant under a rigid rotation on $S^2$ by an angle of $2m\pi/n$
	fixing $\infty$ and $p$ for any $1\le m\le n$ coprime to $n$. We may or may not include the fixed point $p$ 
	in the Cantor set. Apparently this gives $2\varphi(n)$ different conjugacy classes 
	in $\map(S)$ by looking at the rotation number and whether the fixed point $p$ lies in the Cantor set.
	
	Conversely, suppose we have two homeomorphisms $\tilde{g}_i$ on $S^2$ as above 
	fixing $\infty$ and $p_i\neq\infty$
	such that either both $p_i\in K$ or $p_i\notin K$, $i=1,2$. Suppose further that they have the same 
	rotation number $m/n\mod \Z$. Let $q_i:S^2\to S^2/\tilde{g}_i$ be the quotient map. 
	Then $q_i(K)$ is still a Cantor set. There is a homeomorphism 
	$h:S^2/\tilde{g}_1\to S^2/\tilde{g}_2$ taking $q_1(K)$ to $q_2(K)$ and $q_1(\infty)$ to $q_2(\infty)$.
	Moreover, in the case $p_i\in K_i$, 
	we can choose $h$ so that $h(q_1(p_1))=q_2(p_2)$. Then the map
	\[h\circ q_1 : S^2\setminus\{\infty,p_1\}\to (S^2/\tilde{g}_2)\setminus\{q_2(\infty),q_2(p_2)\}\] 
	lifts to $S^2\setminus\{\infty,p_2\}$, which extends uniquely to a map $\tilde{h}: S^2\to S^2$. The map 
	$\tilde{h}$ preserves the Cantor set $K$, satisfies $\tilde{h}(\infty)=\infty$, $\tilde{h}(p_1)=p_2$, 
	and fits into the following commutative diagram.
	\[
	\begin{CD}
	S^2 @>{\tilde{h}}>> S^2\\
	@V{q_1}VV @V{q_2}VV\\
	S^2/\tilde{g}_1 @>{h}>> S^2/\tilde{g}_2
	\end{CD}
	\]
	%It follows that $\tilde{h}$ maps $\tilde{g}_1$-orbits to $\tilde{g}_2$-orbits
	
	For any $x_0\in S^2\setminus\{\infty,p_1\}$, let $x_j=\tilde{g}_1^j x_0$ for $0\le j\le n-1$. Fix a short ray 
	$r$ that passes through $x_0$ but not any $x_j$ for $j\neq 0$ such that $\{\tilde{g}_1^j r\}_{j=1}^{n-1}$ 
	are disjoint (except at $\infty$). Such a ray can obtained for instance by lifting a short ray on 
	$S^2/\tilde{g}_1$. Then there is a permutation $\sigma$ on $\{0,1,\ldots,n-1\}$ such that 
	$\tilde{h}(\tilde{g}_1^j x)=\tilde{g}_2^{\sigma(j)} \tilde{h}(x)$ for all $x$ on $r$ and all 
	$0\le j\le n-1$. Since $\tilde{g}_1$ and $\tilde{g}_2$ have the same rotation number and $\tilde{h}$ 
	maps $\{\tilde{g}_1^j r\}_{j=1}^{n-1}$ to $\{\tilde{g}_2^j \tilde{h}(r)\}_{j=1}^{n-1}$ preserving 
	their circular order on the conical circle $S^1_C$, we must have $\sigma=1$ and 
	$\tilde{h}\tilde{g}_1(x_0)=\tilde{g}_2\tilde{h}(x_0)$. Since $x_0$ is arbitrary, we conclude that $g_1$ and 
	$g_2$ are conjugate by the image of $\tilde{h}$ in $\map(S)$.
\end{proof}

\section{Using Torsion to Obstruct Embeddings}
If $S$ is an orientable, infinite-type surface,
$\map(S)$ has a natural nontrivial topology called
the \emph{permutation topology} which agrees with the quotient topology
inherited from $\homeo(S)$.
Let $\mathscr{C}(S)$ denote the set of isotopy classes
of essential, simple closed curves in $S$.
The set $\mathscr{C}(S)$ is countable, and
it follows from work of Hern\'andez Hern\'andez--Morales--Valdez
\cite{HernandezHernandezMoralesValdez}
that the action of $\map(S)$ on $\mathscr{C}(S)$ is faithful.
A neighborhood basis of the identity in $\map(S)$
is given by the sets
\[ \bigcap_{c\in C}\stab(c), \]
where $C$ ranges over the finite subsets of $\mathscr{C}(S)$.
The action of $\map(S)$ on $\mathscr{C}(S)$ exhibits
$\map(S)$ as a closed subgroup of $\sym$, the group of bijections
of a countable set, again given the permutation topology
\cite[Corollary 6]{VlamisNotes}.
A natural question to ask is whether
$\sym$ embeds in any big mapping class group.
The main result of this section is that there is no embedding.
% In fact, since $\sym$ has \emph{ample generics,}
% any group homomorphism from $\sym$ to a separable
% topological group $G$ is \emph{automatically continuous.}
% Closed subgroups of $\sym$, like $\map(S)$, are separable,
% so there is no injective group homomorphism $\sym \to \map(S)$.

\begin{thm}\label{torsion}
If $G$ is a topological group containing a sequence of
nontrivial torsion elements limiting to the identity,
then $G$ does not embed (as a topological group)
in $\map(S)$, for $S$ an
orientable surface whose boundary is empty or a union of circles.
\end{thm}

Examples of such groups $G$ include $\sym$,
$\out(\pi_1(S))$ when $S$ is of infinite type, 
the automorphism group of a rooted tree, and $\operatorname{Homeo}(S^1)$. 
\Cref{torsion} follows from the following proposition,
which exhibits an open neighborhood of the identity
in $\map(S)$ that contains no nontrivial torsion elements.

\begin{prop}\label{prop: torsion}
Let $\gamma_1$, $\gamma_2$, $\gamma_3$ be essential, 
simple closed curves that cobound a pair of pants.
Any finite order element of
\[\stab [\gamma_1]\cap\stab [\gamma_2]\cap \stab [\gamma_3]\]
in the action of $\map(S)$ on $\mathscr{C}(S)$ is the identity.
\end{prop}

\begin{proof}
Suppose $f\in \map(S)$ has finite order, and that $f$
preserves the isotopy class of each $\gamma_i$
as in the statement. By \Cref{maintheorem},
there is a hyperbolic metric on $S$ such that $f$
may be represented by an isometry $\varphi\colon S \to S$.
The isometry $\varphi$ restricts to an isometry
of the pair of pants $P$ 
bounded by the geodesic representatives
of the curves $\gamma_i$. Since $f$ fixes the isotopy
class of each $\gamma_i$, the isometry $\varphi$ must
fix each geodesic representative setwise.
But this implies $\varphi$ restricts to the identity on $P$,
from which it follows that $\varphi\colon S \to S$
is the identity.
\end{proof}

\section{Compact Subgroups are Finite}
The purpose of this section is to prove the following theorem.

\begin{thm}\label{compactisfinite}
	Let $S$ be an orientable surface whose boundary is either empty
	or a union of circles.
	Compact subgroups of $\map(S)$ are finite,
	and locally compact subgroups are discrete.
\end{thm}

Of course, if $S$ is a finite-type surface,
the permutation topology on $\map(S)$ is discrete,
and there is nothing to prove. 
So suppose $S$ is of infinite type.
The permutation topology on $\sym$ is Hausdorff and totally disconnected
\cite[p. 136]{Cameron}, properties which pass to subspaces like $\map(S)$.
A theorem of van Dantzig \cite[Theorem 1.3]{Wesolek}
says that totally disconnected, locally compact groups admit
a neighborhood basis of the identity given by compact, open subgroups.
Finite Hausdorff spaces are discrete,
thus \Cref{compactisfinite} reduces to showing that
compact subgroups of $\map(S)$ are finite.

Let $K \le \map(S)$ be a compact subgroup.
We will show that $K$ is virtually torsion-free
and that every element of $K$ has finite order.
Only the trivial subgroup of such a group is torsion-free,
thus it follows that $K$ itself is finite.

Therefore, the proof of \Cref{compactisfinite} reduces
to the following two lemmas.

\begin{lem}
	If $S$ is an orientable, infinite-type surface,
	compact subgroups of $\map(S)$ are virtually torsion-free.
\end{lem}

\begin{proof}
	It is well-known that compact Hausdorff, totally disconnected groups
	are isomorphic as topological groups to \emph{profinite} groups,
	inverse limits of finite groups.
	Open subgroups of profinite groups are of finite index.
	In fact, if $G = \varprojlim_{i\in I} G_i$ is an inverse limit
	of the inverse system of finite groups $G_i$, 
	then an open neighborhood basis of the identity in $G$
	is given by the kernels of the maps $G \to G_i$ as $i$ varies.

	Thus if $K$ is a compact subgroup of $\map(S)$,
	the intersection $K \cap U$ of $K$ with $U\subset \map(S)$
	any open neighborhood of the identity
	contains a finite-index subgroup of $K$.
	Examples of torsion-free neighborhoods of the identity were
	constructed in \Cref{prop: torsion},
	establishing the lemma.
\end{proof}

\begin{lem}
	If $S$ is an orientable, infinite-type surface
	and $K$ is a compact subgroup of $\map(S)$,
	every element of $K$ has finite order.
\end{lem}

\begin{proof}
	As discussed in the proof of the previous lemma,
	if $K$ is a compact subgroup of $\map(S)$
	and $U$ is an open neighborhood of the identity,
	the intersection $K \cap U$ contains a finite-index subgroup of $\map(S)$.
	Let $\gamma$ be the isotopy class of an essential, simple closed curve.
	An example of such an open neighborhood $U$ is given
	by those elements of $\map(S)$ preserving $\gamma$.
	If $K'$ is a finite-index subgroup of $K$,
	every element $g \in K$ has a positive power lying in $K'$.
	Thus as $\gamma$ varies, we see that $g$ acts periodically
	on the isotopy class of every essential simple closed curve in $S$.

	We claim that an element $g$ of $\map(S)$ which acts periodically
	on the isotopy class of every essential simple closed curve in $S$
	has finite order in $\map(S)$.
	
	First suppose $S'$ is a finite-type subsurface of $\map(S)$.
	We will show that a power of $g$ fixes $S'$ up to isotopy
	and induces the identity element of $\map(S')$.
	Indeed, the Alexander method \cite[Proposition 2.8]{Farb_Margalit}
	for finite-type surfaces implies the existence of a finite collection $C$
	of essential, simple closed curves on $S'$ 
	such that any element of $\map(S)$ fixing pointwise each element of $C$
	and fixing the isotopy class of each boundary curve of $S'$ in $S$
	restricts to the identity element of $\map(S')$.
	An element $g$ as above acts periodically on each isotopy class,
	thus a power $g^k$ of $g$ restricts to the identity element of $\map(S')$.

	Thus if $S'$ is a finite-type subsurface of $S$ there exists a positive integer
	$k$ such that $g^k$ preserves $S'$ up to isotopy
	and induces the identity element of $\map(S')$.
	Choose a diffeomorphism $\varphi \colon S \to S$ representing $g$
	such that $\varphi^k|_{S'}$ is the identity. Then
	\[ S'' = \bigcup_{i=1}^k \varphi^i(S') \]
	is a finite-type subsurface of $S$ which is preserved up to isotopy by $g$.
	Indeed, by induction as in \Cref{exhaustion} there is an exhaustion of $S$
	by finite-type subsurfaces $\varnothing = S_0 \subset S_1 \subset \dotsb$
	such that for each $j$, $S_j$ is $g$-invariant up to isotopy,
	and the restriction of $g$ to $\map(S_j)$ has finite order.
	
	The final step is to show that $g$ itself has finite order.
	Let $k_j$ denote the order of $g$ when restricted to $\map(S_j)$.
	The claim follows once we show that the orders $k_j$ are uniformly bounded.
	Indeed, let $S_i$ and $S_j$ be terms in the invariant exhaustion above
	satisfying $S_i \subset S_j$. Suppose further that $S_i$ contains
	three simple closed curves which cobound a pair of pants.
	As in \Cref{prop: torsion},
	observe that any finite order element of $\map(S_j)$ preserving the isotopy class
	of each of these three curves is the identity of $\map(S_j)$.
	Thus $i \le j$ implies the powers $k_j$ and $k_i$ satisfy $k_j \le k_i$.
	Indeed, we conclude the order of $g$ in $\map(S)$ divides $k_i$.
\end{proof}

\bibliography{bib.bib}
\bibliographystyle{alpha}
\end{document}